\documentclass[12 pt]{article}
\usepackage{amsmath}
\usepackage{amsthm}
\usepackage{graphicx}
\usepackage{bbm,amssymb}
\usepackage[hyperfootnotes=false]{hyperref} 

\def\th{\theta}
\def\om{\omega}

\def\gL{\Lambda}
\def\d{\delta}
\def\Re{{\rm{Re}}\,} 
\def\Im{{\rm{Im}}\,} 
\def\a{\alpha}
\def\t{\tau}
\def\b{\beta}
\def\Ai{{\rm Ai}}
\def\s{\sigma}

\def\k{\kappa}
\def\G{\Gamma}

\def\f{\varphi}
\def\l{\lambda}

\newcommand{\de}{\partial}

\def\L{\mathcal L}

\newcommand\supp{\textup{supp}\, }

\newcommand{\R}{{\mathbb{R}}}
\newcommand{\C}{{\mathbb{C}}}
\newcommand{\Z}{{\mathbb{Z}}}
\newcommand{\N}{{\mathbb{N}}}

\newcommand{\T}{{\mathbb{T}}}

\def\ha{\frac{1}{2}}

\def\pa{\partial}

\def\de{\partial}
\def\h{\hslash}
\def\m{{\bf m}}

\def\ra{\rightarrow}
\def\preuve{\begin{proof}} 

\def\ga{\alpha}
\def\gb{\beta}

\def\e{\varepsilon}

\def\gg{\gamma}

\def\gi{\iota}

\def\gl{\lambda}

\def\gO{\Omega}

\def\gS{\Sigma}
\def\gL{\Lambda}

\newtheorem{defi}{Definition}

\newtheorem{prop}{Proposition}

\newtheorem{theo}{Theorem}

\title{Singularities of the wave trace for the Friedlander model}
\author{Yves Colin de Verdi\`ere \footnote{Institut Fourier,
 Unit{\'e} mixte
 de recherche CNRS-UJF 5582,
 BP 74, 38402-Saint Martin d'H\`eres Cedex (France);
yves.colin-de-verdiere@ujf-grenoble.fr} \\
Victor Guillemin \footnote{MIT; supported in part by NSF grant DMS-1005696} \\
David Jerison \footnote{MIT; supported in part by NSF grant DMS-1069225 and the
Stefan Bergman Trust.}
}

\date{January 21, 2014}

\begin{document}

\maketitle
\renewcommand{\thefootnote}{}
\renewcommand{\thefootnote}{\arabic{footnote}}

\begin{abstract} 
In a recent preprint, \cite{CGJ}, we showed that for the Dirichlet Laplacian
$\Delta$ on the unit disk, the wave trace $\mbox{Tr}(e^{it\sqrt{\Delta}})$, 
which has complicated singularities on $2\pi - \e < t < 2\pi$,
is, on the interval $2\pi < t < 2\pi + \e$, 
the restriction to this interval of a $C^\infty$ function on its
closure. In this paper we prove the analogue of this somewhat
counter-intuitive result for the Friedlander model. The proof for 
the Friedlander model is simpler and more transparent than in the 
case of the unit disk.
\end{abstract}


\section*{Introduction}

Let $\Omega$ be a smooth  bounded strictly convex region in the plane,
and let $0<\lambda_1< \lambda_2 \leq \ldots$ be the eigenvalues,
with multiplicity, of the Dirichlet problem for the Laplacian 
on $\Omega$:
\[
\Delta u_j = \lambda_j u_j\ \ \mbox{on} \ \ \Omega;
\quad u_j = 0 \ \ \mbox{on} \ \ \pa \Omega,\qquad
(\Delta = -(\pa_{x_1}^2+ \pa_{x_2}^2))
\]
In the spirit of \cite{CdV0,Chaz,DG,AM,G-M}) we will define the 
{\em wave trace} of the Dirichlet Laplacian $\Delta$ to be the sum
\begin{equation}
  \label{wavetrace}
  Z (t) = \sum_{j=1}^\infty e^{it \sqrt{\lambda_j}}\, .
\end{equation}
By a theorem of Andersson and Melrose \cite{AM}, $\Re Z(t)$ 
is a tempered distribution whose singular support in $t>0$ is 
the closure of the set
\begin{equation}\label{lengthspectrum}
\L = \bigcup_{k=2}^\infty \L_k
\end{equation}
where $\L_k$ is the set of critical values of the length function given
by  
\begin{equation}  \label{lengthfunction}
|x_1-x_2| + |x_2-x_3| + \cdots + |x_k - x_1|,
\end{equation}
in which $x_1$ \dots $x_k$ are distinct boundary points.

Each critical point 
of the length function \eqref{lengthfunction} defines a
polygonal trajectory consisting of the line segments joining $x_i$ to
$x_{i+1}$ and $x_k$ to $x_1$.  These are, by definition, the {\em closed
geodesics} on $\Omega$, and $\L$ is the set of lengths of these geodesics, 
i.e. the {\em length spectrum} 
of $\Omega$.  Each set $\L_k$ is closed, but there are limiting
geodesics as $k\to\infty$.  These limits are closed curves, called
{\em gliding rays},  that lie entirely in $\partial \Omega$.

Near $T$ in $\L$, the nature of the
singularity of $Z(t)$ is fairly well understood in generic
cases (see \cite{G-M}); but, at least at present, this is not so for $T$ 
in the closure of the length spectrum.  
For example, if $T$ is the perimeter of 
$\Omega$, then there are trajectories forming a convex polygon with $k$ sides 
whose perimeters tend to $T$ as $k\to\infty$. 
In general, the limiting lengths $T$ are integer multiples of the perimeter
of $\gO$, and the set $(T-\e,T) \cap \L$ is infinite for every $\e>0$.  
The cumulative effect of infinitely many singularities could be 
quite complicated and has not been analyzed completely even in simple model 
examples.  However to the right of $T$ one doesn't encounter this
problem.  $(T,T+\e)\cap \L = \emptyset $ for sufficiently small $\e>0$,
so that $\Re Z(t)$ is $C^\infty$ for $t\in (T,T+\e)$.  Therefore, several
years ago the second author posed the question, what is
the asymptotic behavior of $Z(t)$ as $t\to T$ from the right?

In one simple case, the case of the unit disk, we showed in \cite{CGJ}
that $Z(t)$ is $C^\infty$ to the right, i.e., for some 
$\epsilon >0$, is $C^\infty$ on 
$[T,T+\epsilon )$.  We will publish details of that result elsewhere;
however in this article we will discuss another example for which we 
have been able to show that this  assertion is true, and, in fact, 
for which the verification is simpler and more transparent than in the  
case of the disk.  This is the 
operator $L$ introduced by F. G. Friedlander in \cite{F}, defined by 
\begin{equation} \label{Friedl operator}
L=-\pa _x ^2 -(1+x)\pa _y ^2, \ \  x\in [0,\infty), \ \ y \in \R/2\pi\Z
\end{equation}
on the manifold $M= [0,\infty)\times (R/2\pi\Z)$ with Dirichlet boundary conditions
at $x=0$.  It represents the simplest approximation to the Laplace
operator near the boundary of a convex domain.  

The operator $L$ has continuous spectrum on the space of functions
that depend on $x$ alone.  On the orthogonal complement of this space,
the spectrum is discrete, and the eigenfunctions of $L$ are described 
in terms of the Airy function $\Ai$ as follows (see Appendix A 
\eqref{Airy function}).
\begin{prop}\label{Friedl spectrum} 
The Dirichlet eigenfunctions of $L$ on 
\[
\{\f\in L^2(M): \ \int_{\R/2\pi \Z} \f(x,y) \, dy = 0 \}
\]
are the functions 
\begin{equation}\label{Friedl eigenfunctions}
\f_{m,n}(x,y) = \Ai(n^{2/3} x - t_m) e^{iny}, 
\quad m = 1, \, 2,\,  \dots, \quad n= \pm 1,\,  \pm 2,\, \dots;
\end{equation}
and their eigenvalues are 
\begin{equation}\label{Friedl eigenvalues}
\gl(m,n) = n^2 + n^{4/3} t_m,
\end{equation}
in which $0< t_1 < t_2 < \cdots$ are the roots of $\Ai(-t)=0$.  
\end{prop}
We will define the trace of the wave operator associated to $L$ by 
\begin{equation}
\label{Friedl wave trace} 
Z_{F} (t) = \sum_{m=1}^\infty \sum_{n\in \Z\setminus \{0\}} e ^{it \sqrt{\gl(m,n)}}~.
\end{equation}
Our main result is as follows.
\begin{theo}  \label{Friedl theorem} 
Let $\L_F$ be the Friedlander length spectrum, as computed
in Section \ref{sec:geodesics}.  

\noindent
a) The limit points of $\L_F$ are $\bar \L_F \setminus  \L_F = 2\pi \N$.

\noindent
b) $Z_F$ is $C^\infty$ in the complement of $\bar \L_F \cup (-\bar \L_F) \cup \{0\}$

\noindent
c) For each $\ell\in\N$, there is $\e_\ell>0$ such that $Z_F$ extends
to a $C^\infty$ function on $[2\pi \ell, 2\pi \ell + \e_\ell)$. 
\end{theo}

The main issue in the proof is to identify appropriate non-classical
symbol properties for the function $\gl(m,n)$.  The Friedlander model is 
significantly easier to treat than other integrable models such as
the disk because the asymptotics of the spectral function 
$\gl(m,n)$ follow easily from well known asymptotics of the Airy function.

The paper is organized as follows. 
In Section \ref{sec:geodesics}, we find the length spectrum of $L$.  
In Section \ref{sec:smoothsymbols}, we show how the Poisson 
summation formula is used to deduce smoothness of the trace when 
the spectral function is a classical symbol.  Next,
we show, in Section \ref{sec:symbols}, that 
the Friedlander spectral function $\gl(m,n)$ belongs
to an appropriate non-classical symbol class of product type.  In 
Section \ref{sec:proof}, we prove the smoothness properties of the 
Friedlander wave trace $Z_F$ (Theorem \ref{Friedl theorem}) 
by dividing the $(m,n)$ lattice into sectors and using 
Poisson summation methods adapted in each sector 
to the symbol properties of $\l(m,n)$.  We conclude, in Section \ref{sec:actions},
by computing the action coordinates of the completely integrable system corresponding
to the operator $L$ and show that they capture the first order properties 
of eigenvalues and eigenfunctions of $L$ in sectors
where $\gl(m,n)$ behaves like a classical symbol.  This final
computation reveals more details about which parts of the spectrum 
are linked to which aspects of the geodesic flow.

In Appendix A, we recall the properties of the Airy function and 
prove Proposition \ref{Friedl spectrum} characterizing the 
spectrum of $L$.  In Appendix B we review the semiclassical analysis
of joint eigenfunctions of commuting operators in 
a general integrable system, showing how to find the parametrization 
of the eigenfunctions by a discrete lattice and how to carry out the computations 
of Section \ref{sec:actions} linking the classical and quantum systems 
in a more general context.

\section{Closed geodesics}
\label{sec:geodesics}

Our first task is to calculate the length spectrum $\L_F$ for
the Friedlander model.  Define the Hamiltonian function $H$ 
as the square root of the symbol of $L$,
\[
H(x,y; \xi, \eta) =  \sqrt{\xi^2 + (x+1)\eta^2}
\]
The integral curves of the Hamiltonian vector field of $H$ restricted
to the hypersurface $H=1$ are defined as solutions to 
\begin{equation}\label{hamiltonian equation}
\dot x = \xi, \quad \dot y = (1+x)\eta, \quad \dot \xi = -\eta^2/2, \quad \dot \eta = 0
\end{equation}
with the constraint
\begin{equation}\label{constraint}
\xi^2 + (1+x)\eta^2 = 1
\end{equation}
We begin by analyzing a single trajectory between boundary points.
\begin{prop}\label{prop: reflection}  Consider an 
integral curve of the Hamiltonian system 
\eqref{hamiltonian equation}-\eqref{constraint} on $0\le t \le T$, 
with
\[
x(0) = 0, \quad x(t) >0 \ \mbox{for} \ 0 < t < T, \quad x(T) = 0
\]
If the initial velocity is $\xi(0) =\xi_0>0$, $\eta(0)= \eta_0$, then
final velocity is given by 
\[
\dot x(T) = -\xi_0, \quad \dot y(T) = \eta_0, 
\]
and 
\[
T = 4\xi_0/\eta_0^2;\quad y(T)-y(0) = 4\xi_0/\eta_0 + (8/3)\xi_0^3/\eta_0^3
\]
\end{prop}
\begin{proof} 
Solving the equations, we get
\[
\eta(t) = \eta_0; \quad \xi(t) = \xi_0 - \eta_0^2t/2, \quad  
x(t) = \xi_0 t - \eta_0^2t^2/4
\]
By translation invariance in $y$, we may as well assume $y(0)=0$.  Then
\[
y(t) = \eta_0t + \eta_0\xi_0 t^2/2 - \eta_0^3 t^3/12
\]
and 
\[
x(T) = 0 \implies \xi_0T - \eta_0^2t^2/4 = 0 \implies T= 4\xi_0/\eta_0^2
\]
Substituting for $T$ we obtain
\[
\dot x(T) = \xi_0  - \eta_0^2 T/2 = -\xi_0,
\]
and 
\[
y(T) - y(0)= y(T) = \eta_0 T + \eta_0\xi_0 T^2/2 - \eta_0^3T^3/12 = 
4\xi_0/\eta_0 + (8/3) \xi_0^3/\eta_0^3
\]
Finally, $x(T)=0$ and $\eta(T) = \eta_0$ imply
\[
\dot y(T) = (1 + x(T))\eta(T) = \eta_0
\]
which concludes the proof. 
\end{proof}

We call a continuous curve $(x,y): [a,b]\to  M$ a {\em geodesic}
if there are finitely many points $t_0< t_1 < \cdots t_k$ for which the
curve meets the boundary (i.~e. $x(t_j)=0$) 
and for $t_j < t < t_{j+1}$, the curve is the projection onto
$(x,y)$ of a solution to \eqref{hamiltonian equation} and \eqref{constraint},
and, finally, at each $t_j$ the curve undergoes a reflection in 
the boundary:
\begin{equation}\label{reflection}
\lim_{t\to t_j^+} \dot x(t) 
= - \lim_{t\to t_j^-} \dot x(t); \quad 
\lim_{t\to t_j^+} \dot y(t)  = \lim_{t\to t_j^-} \dot y(t)  \, .
\end{equation}
The Hamiltonian equations are translation invariant in $t$ and
$y$.  Proposition \ref{prop: reflection} shows that a geodesic
with starting velocity $(\dot x(t_j^+),\dot y(t_j^+))=(\xi_0,\eta_0)$ 
has final velocity $(\dot x(t_{j+1}^-),\dot y(t_{j+1}^-)) = (-\xi_0,\eta_0)$.
By \eqref{reflection}, the reflected velocity 
$(\dot x(t_{j+1}^+),\dot y(t_{j+1}^+)) = (\xi_0,\eta_0)$, the
same as the starting velocity.  Consequently, each arc of a geodesic 
between encounters with the boundary is a translate of the others.    
This is analogous to equilateral polygonal trajectories
in the disk.   The figure below depicts a trajectory reflected
at the boundary $x=0$.  Note also that it is tangent to
the caustic, $x= \xi_0^2/\eta_0^2$, depicted by a dotted line.

\bigskip

\begin{center}
\includegraphics[height=.5\textheight]{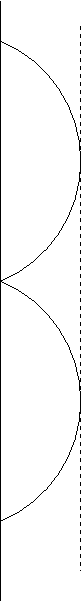} 
\end{center}

\bigskip

\begin{prop} \label{prop: length spectrum}  A closed geodesic with
$k=1, \, 2, \dots $ reflections and winding number $\ell = \pm 1, \, \pm 2, \dots$ 
in the $y$ variable, has length $L_{k,\ell}>0$ determined by the equations
\begin{equation} \label{implicit length formula}
L_{k,\ell}\frac{\eta_0^2}{\sqrt{1-\eta_0^2}} = 4k,\quad
L_{k,\ell} \left( \frac13 \eta_0 + \frac2{3\eta_0}\right) = 2\pi \ell
\end{equation}
\end{prop}
\begin{proof} Proposition \ref{prop: reflection}
shows that the length of a closed geodesic with $k$ reflections
is $kT$, with $T$ given as in the proposition.  Hence, 
using $\xi_0^2 + \eta_0^2 = 1$, 
\[
L_{k,\ell} = kT = 4k\xi_0/\eta_0^2 \implies 
L_{k,\ell}\frac{\eta_0^2}{\sqrt{1-\eta_0^2}} = 4k \, .
\]
Moreover, if the winding number is $\ell$, then if $y(t)$ is
as in Proposition \ref{prop: reflection} with $y(0)=0$, 
\[
k y(T) = 2\pi \ell \implies
k(4\xi_0/\eta_0 + (8/3)\xi_0^3/\eta_0^3) =  2\pi \ell
\]
Substitituting $L_{k,\ell} = 4k\xi_0/\eta_0^2$ into the last equation,
it can be written 
\[
L_{k,\ell}(\eta_0 + (2/3)\xi_0^2/\eta_0) = 2\pi \ell
\iff L_{k,\ell}\left( \frac13 \eta_0 + \frac 2{3\eta_0}\right)  = 2\pi \ell
\]
\end{proof}

The Friedlander length spectrum is defined by
\[
\L_F = \{L_{k,\ell}: k \in \N,\ \ell \in \Z - \{0\} \}
\]
where $L_{k,\ell}$ solve the equations of Proposition \ref{prop: length spectrum}
for some $\eta_0$, $-1 < \eta_0 < 1$.   Note  that if $\eta_0 $ is replaced 
by $-\eta_0$, then $\ell$ is replaced by $-\ell$ and $L_{k,-\ell}= L_{k,\ell}>0$.  
Thus we may confine our attention to $(k,\ell)\in \N \times \N$.  

The
implicit definition of $L_{k,\ell}$ can be restated as follows. Let 
\[
H:\R_+ \times \R_+ \to \R
\]
be the homogeneous function of degree $1$ that is equal to $1$ on
the curve
\begin{equation}\label{curve 1}
u \ra \left(\frac{\pi }{2}\frac{ u^2}{\sqrt{1-u ^2}}, 
\left( \frac{u}3 + \frac{2}{3u} \right)\right) \quad 0 < u < 1.
\end{equation}
Then $L_{k,\ell} = H(2\pi k,2\pi \ell)$.  

\begin{prop}\label{length structure}  The Friedlander
length spectrum $\L_F$ can be written
\begin{equation}\label{Fr length spectrum}
\L_F = H(2\pi \N\times 2\pi\N)
\end{equation}
Furthermore, 

a) $\bar \L_F = \L_F \cup 2\pi \N$.

b) For each $\ell \in \N$, there is $\e_\ell>0$ such that
$(2\pi \ell, 2\pi \ell + \e_\ell) \cap \L_F = \emptyset$.

c)  $\L_F \cap 2\pi \N = \emptyset$
\end{prop}
\begin{proof} Because the curve \eqref{curve 1} is the graph 
of a decreasing function, $L_{k,\ell} < L_{k+1,\ell}$.  Furthermore,
$L_{1,\ell}\to \infty$ as $\ell\to\infty$. Therefore, every
sequence of distinct values of $L_{k,\ell}$ tends to 
infinity unless $\ell$ remains bounded.  It follows from
this and the fact that the curve \eqref{curve 1} is asymptotic
to the horizontal line of height $1$ that all the limit
points are given by
\[
\lim_{k\to \infty} L_{k,\ell} =  
2\pi \ell \, \lim_{k\to \infty} H(k/\ell, 1) = 2\pi \ell
\]
The sequence $L_{k,\ell}$ is increasing in $k$ for fixed $\ell$,
so (a) and (b) are established.  Finally, to confirm (c), note that 
if $2\pi m = L_{k,\ell}$ for some $(k,\ell)$, then the equations 
\eqref{implicit length formula} imply that $\pi$ is algebraic, a 
contradiction. 
\end{proof}
Proposition \ref{length structure} (a) and (c) imply 
Theorem \ref{Friedl theorem} (a).

\section{Poisson summation formula with classical symbols}
\label{sec:smoothsymbols}

The symbol class $S^m(\R^n)$ is the set of
functions $a\in C^\infty(\R^n)$ satisfying
\begin{equation}\label{symbol def 1}
|\de_\omega^\alpha  a(\omega)| \le C_\alpha (1+|\omega|)^{m-\a_1-\a_2-\cdots -\a_n},\quad \om\in \R^n
\end{equation}
for all multi-indices $\a = (\a_1,\a_2,\dots, \a_n)$ and some constants
$C_\a$.  $S^m_{cl}(\R^n)$ consists of the symbols in $S^m(\R^n)$ 
with an asymptotic expansion,
that is, there are homogeneous functions $a_j$ of degree $m-j$,
smooth except at the origin, and a cut-off function
$\chi_N\in C^\infty(\R^n)$ such that $\chi_N = 1$ on $|\om| \ge C_N$, 
and
\begin{equation}\label{symbol def cl}
\left(a - \sum_{j=0}^N a_j\right)\chi_N  \in S^{m-N-1}(\R^2)
\end{equation}
The function $a_0$ is the principal symbol of $a$. 
Functions defined on a cone $\G \subset \R^n$  belong to the symbol
class $S^m(\G)$ ($S^m_{cl}(\G)$) if they are restrictions of
symbols in the class $S^m(\R^n)$ ($S^m_{cl}(\R^n)$).

Let $F$ be a real-valued symbol of first order, $F\in S_{cl}^1(\R^2)$
with positive principal symbol.  
Let $\T^2$ be the 2-torus, $\R^2 /2\pi \Z^2$, and let $P: C^\infty
(\T^2) \to C^\infty (\T^2)$ be the constant coefficient self-adjoint
first order elliptic pseudodifferential operator given by
\[
P(e^{2\pi i(mx + ny)}) = F(m,n) e^{2\pi i(mx+ny)}
\]
The eigenvalues of $P$ are $F(m,n)$, $(m,n) \in \Z^2$, and hence
its wave trace is
\begin{equation}
  \label{smoothZ}
Z(t) =   \sum_{(m,n)\in \Z^2} e^{it F (m,n)} \, .
\end{equation}
Since the summands are smooth functions of $t$, the singularities 
of  $Z(t)$ depend only on the asymptotic behavior of $F(m,n)$ as 
$(m,n)\to \infty$.   
For simplicity we will assume that $F=F_1$ on $\xi^2 + \eta^2 
\ge  1/2$, where $F_1$ is homogeneous of degree $1$ and positive.
(With small modifications the proof we will give applies in
general.)

To identify the length spectrum, consider the Hamiltonian flow on the
cotangent bundle of $\T^2$ associated with the symbol, $F_1$, 
defined by the ordinary differential equations
 \begin{equation}
   \label{smoothflow}
   \Dot{x} = \partial_\xi F_1, \
  \Dot{y} =\partial_\eta F_1 , \quad 
   \Dot{\xi} = \Dot{\eta} =0,
 \end{equation}
Its periodic trajectories are curves of the form, $(x,y) +t
\nabla F_1 (\xi,\eta)$, $0\le t \le T$, satisfying 
\begin{equation}
  \label{smoothlengthspec}
  T\nabla F_1 (\xi ,\eta) = 2\pi (p,q)
\end{equation}
for some $T>0$ and some $(p,q) \in \Z^2$.  Let $T_{p,q}$ be the set
of $T$  for which \eqref{smoothlengthspec} has a solution
for some $(\xi ,\eta)\neq(0,0)$.  (Note that $\nabla F_1$ is homogeneous
of degree $0$, so that we may as well assume $\xi^2 + \eta^2 = 1$.)
The length spectrum is given by 
\[
\L = \bigcup_{(p,q)\in \Z^2}  T_{p,q}
\]
Note that because $\nabla F_1$ is nonzero, bounded, and continuous
on $\xi^2 + \eta^2 = 1$, the length spectrum
is closed and bounded away from $0$. 
Our goal is to prove that the function $Z(t)$ in 
\eqref{smoothZ} is $C^\infty$ in $t>0$ on the complement of the length 
spectrum.

To do so we will apply the Poisson summation
formula to the function, $e^{itF (\xi,\eta)}$.  The Fourier
transform of this function is
\begin{equation*}
  \int_{\R^2} e^{i(tF (\xi ,\eta) -2\pi (x\xi + y\eta))}\, d\xi d\eta
\end{equation*}
so by the Poisson summation formula,
\begin{equation}
  \label{poissonsumformula}
Z(t) =   \sum_{(p,q)\in \Z^2}  Z_{p,q} (t) 
\end{equation}
where
\begin{equation}
  \label{poissonsummand}
  Z_{p,q} (t) = \int_{\R^2} e^{i (tF(\xi,\eta)-2\pi (p\xi + q \eta))}
      \, d\xi d\eta \, .
\end{equation}
We will first show that on the complement of the length spectrum
these individual summands are $C^\infty$.  If $t_0>0$ and $t_0\notin \L$ 
then there exists an interval, $t_0 - \epsilon \leq t
\leq t_0 +\epsilon$, such that for all $|(\xi,\eta)|\ge 1$, either $t
\partial_\xi F(\xi,\eta) - 2\pi p$ or $t
\partial_\eta F-2\pi q$ is non-zero.  Hence for $t$
on this interval $Z_{p,q}(t)$ can be written as a sum
\begin{equation}
  \label{smoothsplitsum}
  \sum_{r=1}^3 \int_{\R^2} e^{i(tF-2\pi(p\xi+q,n))} 
       \rho_r (\xi, \eta) \, d \xi d\eta
\end{equation}
where $\rho_1$  and $\rho_2$ are smooth homogeneous functions of degree
zero on  $|(\xi,\eta)|\ge 1$ and the first of
the functions above is non-zero on $\supp \rho_1$ and the second
non-zero on $\supp \rho_2$.  Finally $\rho_3$ is supported in 
a compact neighborhood of the origin.  The third integrand
is compactly supported, and the corresponding
integral converges along with all derivatives with respect to $t$,
regardless of the range of $t$.   Let
\begin{equation}
  \label{a0def}
  a_0 (\xi ,\eta, p, t) =
     i \left( t \partial_\xi F- 2\pi p \right)^{-1}\,.
\end{equation}
Then
\begin{equation*}
  e^{i(tF-2\pi (p\xi + q\eta))} \rho_1 =
  a_0 \partial_ \xi
     \left( e^{i (tF-2\pi (p\xi + q \eta))} \right) \rho_1
\end{equation*}
hence by integration by parts the first summand of
(\ref{smoothsplitsum}) can be written as
\begin{equation*}
  \int_{\R^2} e^{i(tF -2\pi (p\xi + q\eta))} a_1 \, d\xi d\eta
\end{equation*}
or inductively as
\begin{equation*}
   \int_{\R^2} e^{i(tF -2\pi (p\xi + q\eta))} a_N \, d\xi d\eta
\end{equation*}
where
\begin{equation}    \label{aNdef}
a_1 = -\partial_ \xi (a_0 \rho_1);
\quad
a_N (\xi,\eta,p,t) = -\partial_ \xi(a_0 a_{N-1})\,.
\end{equation}
Since $a_0$ is a homogeneous function of degree zero in
$(\xi,\eta)$ outside a small neighborhood of the origin, it
follows by induction from (\ref{aNdef}) that $a_N$ is a
homogeneous function of degree $-N$, and hence
\begin{equation}
  \label{aNsymbol}
  |a_N (\xi,\eta,p,t)| \leq C_N (1+|\xi|^2 + |\eta|^2)^{-N/2}
\end{equation}
uniformly on the interval $-\epsilon + t_0 \leq t \leq \epsilon + t_0$.
Differentiation $k$ times with respect to $t$ introduces the factor 
$F^k$ in the integrand, which is dominated by $|(\xi,\eta)|^{k}$; hence
the integral converges.   In all, the first of the summand
in (\ref{smoothsplitsum}) is a $C^\infty$  function of $t$ in the range 
specified.  The proof for the term $\rho_2$ is similar,
replacing $a_0$ by $b_0=i\left( t
  \partial_ \eta F - 2\pi q \right)^{-1}$,  $a_1$ by
$b_1 = -\partial_ \eta \, (b_0 \rho_2)$ and $a_N$ by $b_N =
-\partial_ \eta\, (a_0 \, b_{N-1})$ one concludes by the
same argument that the
second of these summands is $C^\infty$ using
integration by parts in $\eta$.   

For $|p|$ large, $t\partial
F/\partial \xi - 2\pi p$ is non-vanishing for all $(\xi,\eta)$ and
all $t$ in $t_0 - \epsilon \leq t \leq t_0 +
\epsilon$, and integration by parts with respect to $\xi$ yields
\[
Z_{p,q}(t) = \int_{\R^2} e^{i(tF-2\pi(p\xi + q\eta))} \tilde a_N \, d\xi d\eta,
\quad N=1,2,\ldots
\]
with $\tilde a_0 = 1$ and 
\[
\tilde a_N = - \partial_\xi (a_0\tilde a_{N-1})
\]
The estimate \eqref{aNsymbol} is replaced by 
\begin{equation}\label{tildeaNsymbol}
   |\tilde a_N (\xi,\eta,p,t)| \leq C_N (1+ |\xi|^2+|\eta|^2)^{-N/2}|p|^{-(N+1)},\quad N\ge 1
\end{equation}
Similarly, for $|q|$ large, integration by parts with respect to $\eta$ yields
an analogous formula with better decay as $q\to\infty$. 
Hence if we break the sum $\sum Z_{p,q}$ into subsums, $\sum
Z_{p,q}$, $|p| < |q|$, and $\sum Z_{p,q}$, $|p| \geq |q|$, and
make use of these estimates we conclude that not only are the
individual summands of \eqref{poissonsumformula} $C^\infty$ but that 
the sum itself is $C^\infty$ on the complement of the length spectrum 
$\L$.

\section{Non-classical symbols}
\label{sec:symbols}

Let $\G$ be a cone (sector with vertex at the origin) in $\R^2$.  
\begin{defi}
A smooth function $a: \G \to \C $
belongs to $\Sigma ^{\ga, \gb}(\G) $ if there exists constants
$C_{j,k} $ so that for all $(\xi,\eta)\in \G$, 
the following upper bounds hold
\[ \forall j,k \geq 0,~|\pa _\xi^j \pa _\eta ^k
a(\xi, \eta )|\leq C_{j,k} (1+ |\xi|)^{\ga -j}(1+ |\eta|) ^{\gb -k} ~.\]
\end{defi}
A typical example is $b(\xi) c(\eta) \in \Sigma ^{\ga,\gb}(\R^2)$ if
$b\in S^\a(\R)$ and $c\in S^\b(\R)$.

\begin{prop} \label{symbol properties}  The following properties hold.

a) $S^0(\G)  \subset \Sigma^{0,0}(\G)$.

b) If $a\in \Sigma^{\ga, \gb }(\G)$ and $b\in \Sigma^{\ga', \gb '}(\G)$,
the product $ab$ belongs to 
$ \Sigma^{\ga +\ga', \gb+\gb' }(\G)$.

c) If $a\in  \Sigma^{\ga, \gb }(\G)$ and
$a\geq C (1+ |\xi|)^{\ga }(1+|\eta|) ^{\gb} $ with $C>0$ 
in $\G$ 
($a$ is said to be {\rm elliptic positive}),
then  $a^{-1}\in  \Sigma^{-\ga, -\gb }(\G)$.  Moreover,
$(a+c)^{-1}\in \Sigma^{-\a,-\b}(\G)$ uniformly for all constants $c\ge 0$. 
\end{prop}

To establish the symbol properties of the phase $\sqrt{\l(m,n)}$ in
the formula \eqref{Friedl wave trace} for the Friedlander wave trace  $Z_F$,
we will need a well known asymptotic formula
for the roots of the Airy function. 
(A proof is given in Appendix A; see also \cite{AS} page 450.)
\begin{prop} \label{Airy roots} There is a symbol $\t\in S^{2/3}_{cl}(\R_+)$ such that
$\t(\xi)>0$ for all $\xi>1/2$,  
\[
t_m = \t(m), \quad m= 1, \, 2, \dots
\]
are the roots of $\Ai(-t)=0$, and the principal symbol of $\t(\xi)$ is
$(3\pi \xi/2)^{2/3}$ ($\xi\to \infty$).
\end{prop}

Denote
\[
F(\xi,\eta) =  (\eta^2 + \eta^{4/3} \t(\xi))^{1/2}
\]
Then $\sqrt{\l(m,n)} = F(m,|n|)$ and the sum 
\eqref{Friedl wave trace} can be written
\[
Z_F(t) = 2\sum_{(m,n)\in \N \times \N} e^{itF(m,n)}
\]
Let $0 < \k_1 < \k_2 <\infty$.  The first quadrant is covered by the union 
of three overlapping sectors
\begin{align*}
\G_1 &= \{(\xi,\eta): |\xi| < \k_1 \eta\}; \\
\G_2 &= \{(\xi,\eta): (\k_1/2) \eta < \xi< 2\k_2 \eta\} \\
\G_3 &= \{(\xi,\eta): \k_2| \eta |< \xi\}
\end{align*}
One can deduce easily from Proposition \ref{Airy roots} the following.
\begin{prop}  \label{Friedl symbol}  Let $G(\xi,\eta) = F(\xi,\eta) - \eta$.  Then 

\noindent
a)  $G\in \Sigma^{2/3,1/3}(\G_1)$, and $\de_\eta G$ is positive,
elliptic in $\Sigma^{2/3,-2/3}(\G_1)$.

\noindent
b)  $F \in S^1_{cl}(\G_2)$ with principal symbol $F_0$ given by 
\[
F_0(\xi,\eta) =  \sqrt{\eta^2 + (3\pi \xi/2)^{2/3}\eta^{4/3}}
\]

\noindent
c)  $F\in \Sigma^{1/3,2/3}(\G_3)$, and  $\de_\xi F$ is positive elliptic
in $\Sigma^{-2/3,2/3}(\G_3)$.
\end{prop}

\section{Poisson summation formula with non-class\-ical symbols}
\label{sec:proof}

We are now ready to finish the proof of 
Theorem \ref{Friedl theorem}.  First, we divide the sum representing
$Z_F$ into three regions corresponding to different behaviors of
the symbol $F(\xi,\eta)$. 
Let $\chi_j\in S^0(\R^2)$ be supported in $\G_j$ and such that
\[
\sum_{j=1}^3 \chi_j(\xi,\eta) = 1 \quad \mbox{for all } \  \xi>0, \ \eta>0, \ \xi^2 + \eta^2 \ge 1/4
\]
Let $\psi \in \C^\infty(\R)$ satisfy $\psi(\xi) = 1$ for all $\xi\ge 1$ 
and $\psi(\xi)= 0$ for all $\xi\le 1/2$.  Then
\[
Z_F(t) = 2\sum_{j=1}^3 I_j(t), \quad
I_j(t) = \sum_{(m,n)\in \Z^2} e^{itF(m,n)} \psi(m)\psi(n) \chi_j(m,n)
\]

The method of Section \ref{sec:smoothsymbols} applies in
the sector $\G_2$.  The times $t$ at which singularities can 
occur are the stationary points $T_{p,q}$ of the phases of the 
representation using the Poisson summation formula, namely,
solutions to 
\begin{equation}\label{stationary}
T_{p,q} \nabla F_0(\xi,\eta) = 2\pi (p,q)
\end{equation}
Furthermore, on $F_0(\xi,\eta) = 1$, 
\[
\nabla F_0(\xi,\eta) = \left(\frac\pi{2} \frac{\eta^2}{\sqrt{1-\eta^2}}, 
\frac13 \eta + \frac 23 \eta^{-1}\right)
\]
Thus if $(k,\ell) \in \N\times \N$, then $T_{k,\ell} = L_{k,\ell}$, exactly
the points of the length spectrum $\L_F$.  The set of $T_{p,q}$ is
\[
-\L_F \cup \{ 0 \} \cup \L_F
\]
(Evidently, $T_{0,0} = 0$ is a stationary point; there are 
no roots $T_{0,q}$, $q\neq0$ and no roots $T_{p,0}$, $p\neq0$ 
of \eqref{stationary}.  The rest of the range $(p,q)\in \Z\times \Z$ 
is covered by the symmetries  $T_{k,-\ell} = T_{k,\ell}$ 
and $T_{-k,-\ell} = -T_{k,\ell}$.)

The restriction to $\G_2$ gives a further limitation on the set of
stationary points.   
Define a smaller collection of stationary points by 
\[
\L_F(M) = \{L_{k,\ell}: (k,\ell) \in \N^2, \quad 1\le k \le M\ell\}
\]
$\L_F(M)$ is a closed, discrete set in $\R$.  
For $M$ sufficiently large depending on $\k_1$,
all the stationary values that occur in the 
the range $\xi \ge \k_1|\eta|/2$ are in $\L_F(M)$.
Because $\L_F(M)$ has no limit points, and the symbol
of the phase has classical behavior in $\G_2$, the method of 
Section \ref{sec:smoothsymbols} applies and yields the following.
\begin{prop}\label{I2smooth} For $M$ sufficiently large
depending on $\k_1$, $I_2$ is smooth in the complement of
\[
\L_F(M) \cup (-\L_F(M)) \cup \{0\}
\]
In particular, for any $\ell\in \N$ and any $\k_1>0$ there is $\e>0$ such
that 
\[
I_2 \in C^\infty((2\pi \ell - \e, 2\pi \ell + \e))
\]
\end{prop}

All the remaining singularities 
in $\bar \L_F -\L_F(M)$ are accounted for in $\G_1$, associated with 
$I_1$.   The main issue is the limit points at $t= 2\pi\ell$.   
Our theorem in $\G_1$ is stated in sufficient 
generality to apply as well to the case of the disk.

\begin{theo}\label{mainbyparts}
Let $\psi(\xi)$ be a classical symbol of degree $0$ in one variable
supported in $\xi \ge 0$ and let $\chi(\xi,\eta) $ be a classical symbol 
of degree $0$ supported in a cone $\G_1$ given by $|\xi  |\leq \k_1 \eta $.
Suppose that the distribution $Z$ is defined by
\[ 
Z(t)=\sum_{(m,n)\in \Z^2}  \psi(m) \chi(m,n) e^{it (n + G(m,n))}~,
\]
where  $G$ belongs to $\Sigma ^{2/3,1/3}$,  
with $\pa _\eta G >0 $  is  positive elliptic in  $\Sigma ^{2/3,-2/3}$
on the support of $\psi(m)\chi(m,n)$.   For any $T>0$ there
exist $\e>0$ and $\k_1>0$ sufficiently small depending on 
$T$ and the constant in the symbol upper bound on $\de_\eta G$
such that 

\noindent
a) if $T = 2\pi \ell$, $\ell\in \N$, then $Z$ is smooth on $[T,T+\e)$;

\noindent
b) if $T\notin 2\pi \N$, then $Z$ is smooth on $(T-\e,T+\e)$.
\end{theo}


\begin{proof}
In the proof of part (a) we will assume for simplicity that $\ell=1$.  The
other cases are similar.  
Let $F(\xi,\eta)=\eta + G(\xi,\eta ) $.   The Poisson summation formula says
\[
Z(t) =  \sum_{(p,q)\in \Z^2} Z_{p,q}(t)
\]
with 
\[ 
Z_{p,q}(t)= \int _{\R^2} \psi(\xi)\chi ( \xi,\eta ) e^{i(t
  F(\xi,\eta) -2\pi (p\xi + q\eta ))}d\xi d\eta ~,
  \]
 Denote  
$Q=i(t
  F(\xi,\eta) -2\pi (p\xi + q\eta ))$ and 
define two differential operators in order to integrate by parts
as follows.
\[ 
Af= e^{itF  (\xi, \eta)  }
\frac{i}{2\pi }\pa _\xi( e^{-itF  (\xi, \eta)}f)= \frac{i}{2\pi }
\pa _\xi f + \frac{t}{2\pi}(\pa_\xi F )f
~.
\]
\[ 
Bf= \frac{1}{i(t \pa _\eta F (\xi, \eta) -2\pi q)}\pa _\eta f 
\]
Then $\displaystyle \frac1p A(e^Q)= B(e^Q)=e^Q $.

Let us start with the most singular case $q=1$.   Because $G\in \gS^{2/3,1/3}$,
we have 
\[
\de_\xi F = \de_\xi G \in \gS^{-1/3,1/3}, \quad \de_\eta G \in \gS^{2/3,-2/3} \, .
\]
The main point is that by Proposition \ref{symbol properties},
\[
\frac{\psi(\xi)\chi(\xi,\eta)}{t\de_\eta F-2\pi} 
= \frac{\psi(\xi)\chi(\xi,\eta)}{t-2\pi + \de_\eta G}  \in \gS^{-2/3,2/3}
\]
uniformly in $t>2\pi$.  Let $A'$ and $B'$ denote the adjoint operators
to $A$ and $B$.  Then Proposition \ref{symbol properties} also implies
that
\[
A' (\Sigma  ^{\ga,\gb  }) \subset  \Sigma ^{\ga -1/3,\beta +1/3};
\quad B'(\Sigma ^{\ga,\beta } )\subset \de_\eta \gS^{\a-2/3,\b+2/3} 
\subset \Sigma ^{\ga -2/3,\gb -1/3}
\]
with constants in the estimates independent of $p$.
\[ 
Z_{p,1}(t)=p^{-M}\int e^Q(A')^M (B')^N ( \psi(\xi) \chi (\xi,\eta ))
 d\xi d\eta ~.\]
Since $\psi(\xi)\chi(\xi,\eta)\in \gS^{0,0}$, the symbol 
$p^{-M}(A')^M ( B')^N ( \psi(\xi) \chi (\xi,\eta ))$
belongs to
$\Sigma ^{ - (M+2N)/3,(M- N)/3 }$ with constants $O(|p|^{-M})$.
Note that symbols in $\gS^{\a,\b}$ supported in the cone $0 \le \xi \le \eta$ are
integrable if $\max(\a,0) + \b< -1$.     Therefore it suffices to take $N\ge M + 5$
to get a convergent integral.    In case $p=0$, take $M=0$ and $N$ large.  
Otherwise, take $M\ge 2$ so  
that the sum over $p$ converges.   So far we have shown that
that sum over $(p,1)$ is bounded.  To prove smoothness in $t$, 
differentiate $k$ times.  
This yields an extra factor $F^k\in \gS^{2k/3,k}$ in the integrand.   
Therefore, 
\[
\sum_{p\in \Z} (d/dt)^k Z_{p,1}(t)
\]
is represented by sum of integrals of symbols 
$\gS^{2k/3 -M/3 - 2N/3, k+M/3 - N/3}$ with constants $O(|p|^{-M})$.  
The integrals converge for $N$ large, for instance  $N\ge M + 3k + 10$.
Again, take $M=0$ when $p=0$ and $M\ge 2$ when $p\neq0$.

Next, consider $q\neq 1$.   The same estimates as before
hold for $A'$.    Recall that we may assume that $\chi$ is zero
on a fixed compact subset.  Provided $t$ is close enough to $2\pi$, and
$\k_1$ sufficiently small, and $|\xi| + |\eta|$ sufficiently large
on the support of $\chi$, we have 
\[
|\de_\eta G|\le C((1+|\xi|)/(1+|\eta|))^{2/3} \le  C\k_1 < < 1
\]
and hence $|t\pa _\eta F -2\pi q| \sim |q-1|$ and 
\[
\psi(\xi)\chi(\xi,\eta)/(t\de_\eta F - 2\pi q) \in \gS^{0,0}, \quad \mbox{with 
constants} \quad  O(1/|q-1|)
\]
Hence, 
\[
B' :\Sigma^{\ga,\gb } \ra \Sigma ^{\a,\b - 1}
\]
with constants $O(1/|q-1|)$.  
So the symbol $p^{-M}(A')^M ( B')^N ( \psi(\xi)\chi (\xi,\eta ))$
belongs to
$\Sigma ^{ -M/3,-N }$ 
with constants $O(|p|^{-M}|q-1|^{-N})$.  The integral converges
provided $N\ge2$.   When $p=0$, take $M=0$; when $p\neq0$ take
$M\ge2$.  Then the sum over $(p,q)$, with $q\neq1$ also converges.
Lastly, if we apply $(d/dt)^k$, we obtain
integrands that belong to
$\Sigma ^{2k/3-M/3,k-N }$  with constants $O(|p|^{-M}|q-1|^{-N})$,
so the symbols are integrable if $N \ge 5k/3 + 2$.  
To sum the series, for $p\neq0$, $q\neq1$, let $M\ge 2$.    
For $p=0$, $q\neq1$,  use $M=0$.  This concludes the proof of part (a) of
Theorem \ref{mainbyparts}.

For part (b), once again for any $\d>0$ we can choose
$\k_1>0$ sufficiently small that on the support of $\chi$ (where
$|\xi|+ |\eta|$ is sufficiently large)
\[
|\de_\eta G| \le C ((1+|\xi|/(1+|\eta|)^{2/3} \le C \k_1^{2/3} < \d
\]
Let $q_0$ be the integer such that $|T-2\pi q_0|$ is smallest, and choose
both $\d>0$ and $\e>0$ less than $|T-2\pi q_0|/4$.  Then for $t\in (T-\e,T+\e)$,
$|t\de_\eta F - 2\pi q_0| \ge |T-2\pi q_0|/4$ on the support of $\chi$.  Hence,
\[
\psi(\xi)\chi(\xi,\eta)/(t\de_\eta F - 2\pi q) \in \gS^{0,0}, \quad \mbox{with 
constants} \quad  O(1/(1 + |q-q_0|))
\]
and the rest of the proof proceeds as in the case $q\neq1$ above.  This
concludes the proof of Theorem \ref{mainbyparts}. 
\end{proof}

The next proposition will take care the integral $I_3$.
\begin{prop}  \label{I3smooth}
Let $\psi \in C^\infty(\R)$ satisfy $\psi(\eta)=1$ for $\eta$ sufficiently large
and $\psi(\eta) = 0$ for $\eta\le 0$. Suppose that where  $F$ belongs to 
$\Sigma ^{1/3,2/3}$ and $\pa _\xi F $  is  positive elliptic in  $\Sigma ^{-2/3,2/3}$
on the cone $0 < \eta < \xi$.  For each $T<\infty$ there exists $\k_2$ such
that if $\chi(\xi,\eta) $ is a classical symbol of degree $0$ supported 
in a cone $\G$ given by $\xi > \k_2 |\eta| $, then  the 
distribution
\[ 
Z(t)=\sum_{(m,n)\in \Z^2}  \psi(n) \chi(m,n) e^{it F(m,n)}~,
\]
satisfies $Z\in C^\infty((0,T))$.
\end{prop}
\begin{proof}  
As usual, the Poisson summation formula implies
\[
Z(t) =  \sum_{(p,q)\in \Z^2} Z_{p,q}(t)
\]
with 
\[ 
Z_{p,q}(t)= \int _{\R^2} \psi(\eta)\chi ( \xi,\eta ) e^{i(t
  F(\xi,\eta) -2\pi (p\xi + q\eta ))}d\xi d\eta ~,
  \]
 Denote  
$Q=i(t
  F(\xi,\eta) -2\pi (p\xi + q\eta ))$ .  
  
 Case 1. $p=0$.  Then  $ Q = i(tF-2\pi q\eta)$ and 
 \begin{align*}
Z_{0,q}(t) =  \int\int e^{Q} \psi(\eta) \chi(\xi,\eta) \, d\xi d\eta
 &
 = - \int\int e^{Q} \de_\xi (\chi/\de_\xi Q)\psi(\eta) \, d\xi d\eta \\
 & =
 \cdots = 
 \int\int e^{Q} a_M(\xi,\eta)\psi(\eta) \, d\xi d\eta 
 \end{align*}
 with $a_0= \chi$, $a_{M+1} = -\de_\xi (a_M/\de_\xi Q)$.   Since
 $\de_\xi F$ is positive elliptic in $\Sigma^{-2/3,2/3}(\G)$,
 \[
 1/\de_\xi Q = 1/it\de_\xi F
 \]
 belongs to $\Sigma^{2/3,-2/3}$ with support on $\G$ and 
 bounds depending on $1/t$.  By induction,  $a_M \in \Sigma^{-M/3,-2M/3}$
 with support in $\G_3$.   Thus $Z_{0,q}$ is represented by a convergent
 integral.  Moreover, if $q=0$, we may differentiate $k$ times with respect to
 $t$ to obtain an integrand dominated by $|a_M F^k|$ which is integrable
 if $k < M/3 - 2$.  For $q\neq0$, use
 \[
 e^{2\pi i q\eta} 
 = \frac1{2\pi i q} \de_\eta e^{2\pi i q\eta} 
 \]
 to integrate by parts $N$ times to obtain
 \[
(d/dt)^j Z_{0,q}(t) =  \frac{\pm 1}{(2\pi i q)^N}
  \int\int e^{Q} \de_\eta^N\left[(iF)^j e^{itF} \psi(\eta) a_M(\xi,\eta)\right] \, d\xi d\eta 
 \]
 The integrand is majorized by $(1+|\xi|)^{j+ (N-M)/3}$ which is convergent in 
 $\G$  for sufficiently large $M$ depending on $j$ and $N$.  (Note
 that because $\chi$ is supported in $\xi  < |\eta|$, 
$\psi'(\eta) \chi(\xi,\eta)$ has compact  support, and every term
 in  which a derivative falls on $\psi$ is convergent.)  The factor $1/q^N$ 
 makes the sum over $q$ convergent as  well.

Case 2.  $p\neq0$.  We use the same integration by parts as in Case 1.
Note that the set of $(m,n)\in \Z^2$ where $\psi(n)\chi(m,n)>0$ and 
$m< \xi_0 $ is finite.  Any finite sum of 
exponentials is smooth in $t$, so we may assume without loss of
generality that $\chi(\xi, \eta) $ is supported on $\xi > \xi_0$
for some large $\xi_0$.  It follows that 
\[
\de_\xi F(\xi,\eta) \le C\xi^{-2/3}\eta^{2/3}
\]
on the support of $\psi(\eta) \chi(\xi,\eta)$.   Furthermore,
if $\k_2$ is sufficiently large, depending on $T$,
\[
0 \le t \de_\xi F \le C T\k_2^{-2/3}  \le \pi, \quad \mbox{for all} 
\quad 0 \le t \le T
\]
Thus 
\[
\chi/\de_\xi Q = 1/i(t\de_\xi F - 2\pi p) \in \Sigma^{0,0}(\G)
 \]
 with bounds $O(1/|p|)$.    By induction, 
 \[
 a_M \in \Sigma^{-M,0}
 \]
 with bounds $O(1/|p|^M)$.   The rest of the proof proceeds as in Case 1.
If $q=0$, then no more integrations  by parts are needed.  If $q\neq0$, 
then one uses integration by parts in $\eta$ as in Case 1 to introduce a factor 
$|q|^{-N}$ so as to be able to  sum in both $p$ and $q$. 
 \end{proof}

Finally, we assemble these ingredients to prove Theorem \ref{Friedl theorem}.
Part (a) was already proved at the end of Section \ref{sec:geodesics}. 
For part(b), choose $T \notin \bar\L_F \cup (-\bar\L_F) \cup\{0\}$, 
and choose $\e$ so that 
\[
[T-\e,T+\e] \cap  (\bar\L_F \cup (-\bar\L_F) \cup\{0\}) = \emptyset
\]
Choose $\k_1$ sufficiently small depending on the symbol bounds and 
the distance from $[T-\e,T+\e]$ 
to $2\pi \Z$ and $\k_2$ sufficiently large depending on the size of $T$, 
we may apply Theorem \ref{mainbyparts} (b), Proposition \ref{I2smooth} and 
Proposition \ref{I3smooth}, respectively, to $I_1$, $I_2$ and $I_3$
to obtain smoothness in $(T-\e,T+\e)$.  Note that
Proposition \ref{Friedl symbol} implies that the correponding phases 
satisfy the appropriate symbol estimates.  For part (c), 
apply Theorem \ref{mainbyparts} (a) to $I_1$, to obtain smoothness
in $[2\pi \ell, 2\pi \ell + \e)$.  By Propositions \ref{I2smooth} and \ref{I3smooth}
$I_2$ and $I_3$ are smooth near $2\pi \ell$ on both sides.  This
concludes the proof of Theorem \ref{Friedl theorem}.

\section{Action variables}
\label{sec:actions}

The classical mechanical system associated to the Friedlander model
has two commuting conserved quantities,
\[
H(x,y,\xi,\eta) =  \sqrt{\xi^2 + (1+x)\eta^2}; \quad J(x,y,\xi,\eta) = \eta.
\]
In this section we will show how the action variables associated to
this system determine the first order asymptotics of the eigenvalues
and eigenfunctions of the Friedlander operator $L$.  (See Appendix B 
for the general framework of this calculation.)

Let $H_0$ and $J_0$ be real numbers satisfying $0 < |J_0| < H_0$. 
The Lagrangian submanifold
\begin{equation}\label{eq:lagrangian}
H = H_0, \quad J = J_0
\end{equation}
is a torus foliated by trajectories with initial 
velocity $(\xi_0,\eta_0)$ on $x=0$,
with $\xi_0>0$, $\xi_0^2 + \eta_0^2 =H_0$ and $J_0 = \eta_0$. 
These trajectories are tangent to the caustic 
\[
x= \frac{\xi_0^2}{\eta_0^2} = \frac{H_0^2-J_0^2}{J_0^2},
\]
depicted as a dotted line in the figure below and in Section \ref{sec:geodesics}.





Fix $H=1$, so that $\xi_0^2 + \eta_0^2 = 1$.  
Denote by $\gamma(t) = (x(t),y(t),\xi(t),\eta(t))$
the trajectory given by Proposition
\ref{prop: reflection} on the interval $-T/2 \le t \le T/2$.
Define by $\gamma_1$ the loop (whose projection onto $\R^+\times S^1$ is
pictured in bold in the figure below)
that follows $\gamma$ on its two curved portions and then
the segment of the caustic $x = \xi_0^2/\eta_0^2$, 
with $\xi=0$, $\eta=\eta_0$ (oriented by $\dot y <0$) 
going from $\gamma(T/2)$ to $\gamma(-T/2)$.  

\bigskip

\begin{center}
\includegraphics[height=.5\textheight]{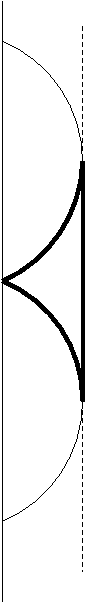} 
\end{center}

\bigskip

Define by $\gamma_2$ the loop that follows
the caustic $x = \xi_0^2/\eta_0^2$, $0\le y \le 2\pi$, $\xi=0$, and
$\eta = \eta_0$, with orientation $\dot y >0$.  

The loops $\gamma_1$ and $\gamma_2$ are a basis
for the homology of the torus $H=1$, $J = \eta_0$, and
the action coordinates are defined by 
\[
I_1 = \int_{\gamma_1} \xi \, dx + \eta\, dy,\quad
I_2 = \int_{\gamma_2} \xi \, dx + \eta\, dy.
\]
Since $\eta(t) = \eta_0$ is constant, $T = 4\xi_0/\eta_0^2$, and
$H = H_0 = \xi_0^2 + \eta_0^2$, we have
\begin{align*}
I_1 
&= \int_{\gamma_1} \xi \, dx = \int_{\gamma} \xi \, dx \\
& = 2\int_0^{T/2} (\xi_0- \eta_0^2t/2)^2 \, dt \\
& = \frac{4\xi_0^3}{3\eta_0^2} = \frac43 (H_0^2- \eta_0^2)^{3/2}{\eta_0^{-2}}
\end{align*}
Moreover,
\[
I_2 = \int_{\gamma_2} \xi \, dx + \eta\, dy = 
\int_{\gamma_2} \eta\, dy = 2\pi \eta_0
\]
More generally, by homogeneity,
\[
I_1 = \frac43 (H^2- J^2)^{3/2}{J^{-2}}; \quad I_2 = 2\pi J.
\]
Thus $H$ and $J$ can be written as functions of the action variables,
\[
H^2 = J^2 + \left(\frac34 I_1J^2\right)^{2/3} = 
 \frac{I_2^2}{4\pi^2} + \left(\frac{3I_1I_2^2}{16\pi^2}\right)^{2/3};
\quad J= \frac{I_2}{2\pi} \, . 
\]
Define $\Lambda(m,n)$ as the values of $H^2$ 
when the action variables belong to the integer lattice
$(I_1,I_2) = 2\pi (m,n)$, $(m,n)\in \Z^2$.  (Since $I_1>0$, 
we consider only $m>0$.)  Then
\[
\Lambda(m,n) = n^2 + (3\pi/2)^{2/3}m^{2/3}n^{4/3}.
\]
The Bohr-Sommerfeld energy levels of $H$ are defined
as the values of $H$ on the lattice, namely $\sqrt{\Lambda(m,n)}$. 

\begin{prop} \label{prop:asympt1}  In any sector of the form 
$0< c_1 \le |n|/m \le c_2$, the eigenvalues $\l(m,n)$ of the
Friedlander operator satisfy
\[
\sqrt{\l(m,n)} = \sqrt{\Lambda(m,n)} + O(1), \quad (m,n) \to \infty
\]
\end{prop}
The proposition is an immediate consequence of the formula for
$\L(m,n)$ and the fact that 
\[
\l(m,n) = n^2 + n^{4/3}\tau(m), 
\]
with $\tau(m) = (3\pi m/2)^{2/3} + O(1)$, $m=1, \, 2,\,  \dots$
(see Proposition \ref{Airy roots}).

Next, we show that the phases of the eigenfunctions are
equal, asymptotically, to the generating function of 
the canonical relation of the Lagrangian submanifold 
\eqref{eq:lagrangian}.  Indeed, if $\rho(x,y)$ is the
generating function, then by definition, the
Lagrangian is the graph of $(x,y)\mapsto (\rho_x,\rho_y)$:
\[
\rho_x^2 + (1+x)\rho_y^2 = H_0^2; \quad \rho_y = J_0 \, .
\]
Thus, 
\[
\rho_x = \pm \sqrt{H_0^2 - J_0^2 - xJ_0^2};\quad 0 \le x \le \frac{H_0^2-J_0^2}{J_0^2},
\]
and 
\[
\rho(x,y) =  \pm \frac23 \frac{(H_0^2 - J_0^2 - xJ_0^2)^{3/2}}{J_0^2} + J_0y
\]
Notice that the upper limit of the range of $x$ is exactly the
caustic
\[
x= \frac{H_0^2-J_0^2}{J_0^2},
\]
so that the graph over this interval in $x$ (and all $y$) covers
the whole Lagrangian submanifold.  

Now let the submanifold be given by $(I_1,I_2) = 2\pi (m,n)$.  Then
\[
H_0^2 = n^2 + \left(\frac32 \pi mn^2\right)^{3/2}; \quad J_0 = n,
\]
and we have two generating functions
\[
\rho_\pm(x,y) = 
\pm \frac23 \left( \left(\frac{3\pi}{2}m\right)^{2/3} - n^{2/3}x\right)^{3/2} + ny.
\]

To compare this with the phase of $\f_{m,n}$, recall that for $t_m > n^{2/3}x$, 
\begin{align*}
\f_{m,n}(x,y) &= \Ai(n^{2/3}x - t_m)e^{iny}  \\
& \sim
\frac12\pi^{-1/2}|n^{2/3}x - t_m|^{-1/4}
\sin \left(\frac23(t_m- n^{2/3}x)^{3/2} 
+ \pi/4\right) e^{iny}\, .
\end{align*}
(See Proposition \ref{Friedl spectrum} and Appendix A.)  Thus $\f_{m,n}$ 
is a sum of two terms with phases 
\[
\pm \left[\frac23 \left(t_m - n^{2/3}x\right)^{3/2} +  \pi/4\right] -\pi/2 + ny
= \rho_\pm(x,y) + O(1),
\]
as $(m,n)\to \infty$ in sectors of the form $0 < c_1 < |n|/m < c_2$ 
(since $t_m = (3\pi m/2)^{2/3}(1+ O(1/m))$).

\section*{Appendix A: Airy functions}\label{app:airy}

An Airy function is a solution  $y(t)$ of the differential
equation $y''=ty$. The classical properties of Airy functions
are treated in \cite{AS}, Section 10.4 and \cite{H}, Section 7.6,
and reviewed briefly here.

The {\it Airy  function} (or {\it Airy integral}), denoted  $\Ai$, was
introduced in 1838  by G.B. Airy in the remarkable paper \cite{Ai}
in order to describe the intensity of the light near a caustic.
This allowed him  to propose a good theory for the rainbow.
 He  defined the function
$\Ai$ by the oscillatory integral
\begin{equation}\label{Airy function}
\Ai(s) = \frac1{2\pi} \int_{-\infty}^\infty e^{i(sx + x^3/3)} \, dx
\end{equation}
In other words,  $\Ai$ is the inverse Fourier transform of 
$x\mapsto e^{ix^3/3}$.

The Airy function $\Ai(t)$ is positive for $t\geq 0$ and decays exponentially 
as $t\ra +\infty$; it is oscillating for $t<0$ with an expansion,
given by the method of stationary phase, of the form
\[
\Ai(t) = \pi^{-1/2}|t|^{-1/4}\left[\sin\left(\zeta + \frac{\pi}4\right) + 
\Re(e^{i\zeta}\alpha(\zeta))\right], \quad t < 0, 
\]
with $\zeta= \frac{2}{3} |t|^{3/2}$ and $\alpha$ a classical symbol of order $-1$. 
Differentiating, we find
\[ {\Ai}'(t) = -\pi^{-\ha}|t|^{1/4}
\left[ \cos \left(\zeta +\frac{\pi}{4}\right) + \Re(e^{i\zeta} \beta(\zeta))\right],
\quad t < 0, 
\]
with $\beta$ a classical symbol of order $-1$.

The function $\Ai$ has infinitely many real negative zeros 
\[ 
\cdots < -t_m < \cdots < -t_1 < 0 \, . 
\]
To derive the asymptotic formula for $t_m$, denote
\[
f(s) = is^{1/3} \Ai(-s^{2/3}) - \Ai'(-s^{2/3}), \quad s> 0.
\]
The foregoing asymptotic expansions can be written
\[
f(s) = \frac{1}{\sqrt{\pi}} s^{1/6}e^{i(2s/3+ \pi/4) + c(s)},
\]
where $c$ is a complex valued classical symbol of order $-1$.
Define
\[
\th(s) = 2s/3 + \pi/4 + \Im c(s) = \arg f(s),
\]
$\Im f(s) >0$ for all $0 < s < t_1^{3/2}$, and specify the branch 
of $\th$ by $0 < \theta(s) < \pi$ in that range.   Thus $\th$ is a classical symbol
of order $1$ with principal part $2s/3$ as $s\to \infty$, and 
\begin{equation}\label{Airy arg}
\tan \th(s) = - \frac{s^{1/3}\Ai(-s^{2/3})}{\Ai'(-s^{2/3})}, \quad s>0.
\end{equation}

\begin{prop}\label{prop:Airy appendix}  Let $\th$ be as above.
Then $\th'(s)>0$ for all $s>1/4$, and there is an absolute 
constant $c_1<1$ such that 
\[
\t(\xi) = [\th^{-1}(\pi \xi)]^{2/3} > 0, \ \mbox{for} \ \xi\ge c_1,
\]
is well defined; $\t\in S^{2/3}_{cl}(\R_+)$ with principal term 
$(3\pi\xi/2)^{2/3}$.  Finally, 
\[
t_m = \t(m), \quad m= 1,\, 2, \, \dots,
\]
where $0> -t_1> -t_2 > \cdots$ are the zeros of $\Ai$. 
\end{prop}
\begin{proof}
\[
\th'(s) =  (1/3|f'|^2)[2s^{2/3} a^2 +  2b^2 -  s^{-2/3}ab], 
\quad a= \Ai(-s^{2/3}), \quad
b = \Ai'(-s^{2/3}).
\]
Therefore, $\th'(s)>0$ for $s> 1/4$.  Because $\th$ is increasing, 
the zeros of $\Ai$ are given by
\[
\th(t_m^{3/2}) = m\pi, \quad m = 1, \, 2, \, \dots
\]
Let $\th(1/4) = \th_0$.  $0 < \th_0 < \th(t_1^{3/2}) = \pi$ because 
$t_1^{3/2} \approx (2.33)^{3/2} > 1/4$.    The inverse function
$\s$ of $\th$ is  a classical symbol of order $1$ on $[\th_0,\infty)$.
Then $\t(\xi) = \s(\pi \xi)^{2/3}$ is defined
for all $\xi \ge \th_0/\pi$ and satisfies all the properties of the proposition.
\end{proof}

Now we turn to proof of Proposition \ref{Friedl spectrum} describing the 
spectrum of the Friedlander operator $L$.  Begin by noting that if $L$ 
is restricted to functions $f(x)$ of $x$ alone, it becomes
\[
Lf(x) = -f''(x),
\]
which has completely continuous spectrum on $L^2([0,\infty))$.  
Next, consider $\f$ in the  orthogonal complement,
\[
\{\f\in L^2(M): \ \int_{\R/2\pi \Z} \f(x,y) \, dy = 0 \}
\]
By separation of variables, the eigenfunctions are of
the form $f(x) e^{iny}$, where $n$ is a non-zero integer, and $f$ satisfies
\begin{equation}\label{eigen equation}
L_nf := (-\de_x^2 + n^2 (1+x)) f = -\l f
\end{equation}
For $n\neq0$, the operator $L_n$ acting on functions satisfying $f(0)=0$
is self-adjoint and has compact resolvent, so a complete system
of eigenfunctions is obtained by solving \eqref{eigen equation}
for $\l$ and $f$ such that $f(0)=0$ and $f\in L^2([0,\infty))$.  
Making the change of variables
\[
f(x) = A(n^{2/3}x -t),
\]
equation \eqref{eigen equation} becomes
\[
n^{4/3}(-A''(s) + sA(s)) = (\l - n^{4/3}t - n^2)A(s) \ .
\]
Choose $t$ so that $\l= n^2 + n^{4/3} t$, then the equation is 
the Airy equation
\begin{equation}\label{Airy equation}
-A''(s) + sA(s) = 0
\end{equation}
Since $f$ is in $L^2([0,\infty))$, the same is true of $A$.  
Up to a constant multiple, there is a unique such function $A$,
the Airy function $\Ai$.   
The Dirichlet condition $f(0)=0$ implies $\Ai(-t)=0$, and the roots
of this equation are $0 < t_1 < t_2 < \cdots$.  Hence 
the formulas in the proposition for the eigenfunctions
and eigenvalues are confirmed.

\section*{Appendix B:  Semiclassical properties of eigenfunctions
and eigenvalues for completely integrable systems}

Let $X$ be an n-dimensional manifold and $P_1$, \dots, $P_n$ be
commuting self-adjoint zero order semiclassical pseudodifferential
operators with leading symbols $p_1$, \dots, $p_n$.   In
Section \ref{sec:actions} we considered the example $X = \R^+ \times S^1$,
\[
P_1 = 
\h^2 (\de_x^2 + (1+x)\de_y^2); 
\quad
P_2 = i \h \de_y.
\]
(These operators are of order zero with respect
to $\de_x$, $\de_y$ and multiplication by $\h$.)

We seek the joint eigenfunctions $\f(x)$, $x\in X$, of $P_j$ in the form 
\begin{equation}\label{eq:eigenform}
\f(x) = a(x,h) e^{i\rho(x)/h}, \quad a(x,h) = a_0(x) + ha_1(x) + h^2 a_2(x) + \cdots,
\end{equation}
solving asymptotically as $\h \to 0$
\[
P_j \f  = \gl_j(h) \f, \quad j= 1,\, \dots, \, n, \qquad
\gl_j(h) = \gl_j(0) + h \gl_{j,1} + h^2 \gl_{j,2} + \cdots \, .
\]

Let $\gL_\rho$ be the graph of $d\rho$ in $T^*X$.   $\gL_\rho$ is a Lagrangian 
submanifold and determines
$\rho$ up to an additive constant; $\rho$ is known
as the {\em generating function} of $\gL_\rho$.
Setting $\h=0$, one finds the {\em eikonal equation},
the first of the heirarchy of equations for $a$, $\rho$, and $\gl_j(h)$:
\begin{equation}\label{eq:eikonal}
\left. p_j \right|_{\gL_\rho} = \gl_j(0) \, . 
\end{equation}

Let us confine ourselves to  the sequence $\h = 1/N$ for integers $N\to \infty$.  
Then \eqref{eq:eigenform}
defines $\f$ provided $\rho$ is well-defined modulo $2\pi \Z$.  To see
what this entails, denote by $\kappa$ the projection, $\gL_\rho \to X$, and
by $\gi$ the inclusion, $\gL_\rho \to T^*X$, and denote
\[
\ga = \sum_{j=1}^n \xi_j dx_j \, ,
\]
the canonical 1-form on $T^*X$.   The fact that $\gL_\rho$ is
the graph of $d\f$ implies
\begin{equation}\label{eq:localcanonical}
d\kappa^* \rho = \gi^* \ga \, .
\end{equation}
If $(\l_1(0),\dots,\l_n(0))$ is a regular value of $(p_1,\dots,p_n)$ and
$\gL_\rho$ is compact, then $\gL_\rho$ is a torus of dimension $n$.  Thus 
if $\gg_1$, \dots, $\gg_n$ is a basis of the first homology group of $\gL_\rho$, 
then \eqref{eq:localcanonical} implies 
\begin{prop}\label{prop:globalcanonical} $\rho$ is well-defined modulo
$2\pi \Z$ if and only if
\[
I_j : = \int_{\gg_j} \ga  \in 2\pi \Z, \quad j= 1, \, 2, \, \dots \, , n.
\]
\end{prop}
The functions $I_j$ are known as action coordinates on $T^*X$.  
Let $\m\in \Z^n$.  In the generic case in which 
\[
(I_1,\dots, I_n) = 2\pi \m
\]
defines an $n$-dimensional Lagrangian torus in $T^*X$, the
torus is known as a Bohr-Sommerfeld surface, and we will denote 
it by $S(\m)$.   

Proposition \ref{prop:globalcanonical} says that the generating 
function $\rho_\m$ of $S(\m)$ is well-defined modulo $2\pi\Z$.
Thus one can follow the ansatz \eqref{eq:eigenform} 
to seek a joint eigenfunction $\f$ whose phase function $\rho$ equals $\rho_\m$
to first order as $\h \to 0$ and whose eigenvalue $\gl_j$ with respect $P_j$ is 
given to first order by $p_j$ evaluated on $S(\m)$ as in 
\eqref{eq:eikonal}.  Under suitable further hypotheses, all but finitely many 
joint eigenfunctions are obtained in this way.  Typically, $\f$ and $S(\m)$ 
exist for a cone of values of $\m\in \Z^n$  (see \cite{CdV2}).

\bibliographystyle{plain}

\end{document}